\documentclass[11pt]{article}
\usepackage{amsmath, amsthm}
\usepackage{amssymb,amsfonts,cite}
\usepackage{enumerate}
\usepackage{tikz}
\usepackage{forloop}
\usepackage{verbatim}

\newcommand{\fd}{\chi_{2}^{f}}

\newtheorem{theorem}{Theorem}[section]

\newtheorem{problem}[theorem]{Problem}

\newtheorem{conjecture}[theorem]{Conjecture}

\newtheorem{conjecture-conclude}{Conjecture}
\newtheorem{problem-conclude}[conjecture-conclude]{Problem}
\newtheorem{claim}{Claim}

\newcommand\bonusspiral{} 
\def\bonusspiral[#1](#2)(#3:#4)(#5:#6)[#7]{
	\pgfmathsetmacro{\domain}{#4+#7*360}
	\pgfmathsetmacro{\growth}{180*(#6-#5)/(pi*(\domain-#3))}
	\draw [#1,
	shift={(#2)},
	domain=#3*pi/180:\domain*pi/180,
	variable=\t,
	smooth,
	samples=int(\domain/5)] plot ({\t r}: {#5+\growth*\t-\growth*#3*pi/180})
}

\begin{document}

\title{\Large Frugal coloring in planar graphs with maximum degree 4 and 5}

\author{
Bo\v{s}tjan Bre\v{s}ar$^{a,b}$  \and P\v{r}emysl Holub$^{c}$ \and Marko Jakovac$^{a,b}$ \and Petra Melicharov\' a$^{c}$
}

\date{\today}

\maketitle

\begin{center}
$^a$ Faculty of Natural Sciences and Mathematics, University of Maribor, Slovenia\\
\medskip

$^b$ Institute of Mathematics, Physics and Mechanics, Ljubljana, Slovenia\\
\medskip

$^c$ University of West Bohemia, Plze\v n, Czech Republic
\\

\end{center}

\begin{abstract}
A proper coloring of vertices of a graph $G$ is called $t$-frugal if no color appears more than $t$ times in the neighborhood of any vertex of $G$. While $1$-frugal coloring coincides with $2$-distance coloring, we study $2$-frugal colorings of planar graphs with maximum degree $4$ and $5$, respectively. We prove that every planar graph with maximum degree $4$ admits a $2$-frugal coloring with $7$ colors, while a planar graph with maximum degree $5$ admits a $2$-frugal coloring with $10$ colors. The results present variations of parts of the famous Wegner's conjecture, which for planar graphs with maximum degree $\Delta\in\{4,\ldots,7\}$ states that their $2$-distance chromatic number is at most $\Delta+5$. 
\end{abstract}


\medskip\noindent
{\bf AMS Subj.\ Class: 05C15, 05C12}

\section{Introduction}
\label{sec:intro}

A {\em $2$-distance coloring} of a graph $G$ is a coloring of the vertices of $G$ such that vertices at distance at most two obtain distinct colors. Alternatively, it is a proper coloring of the square $G^2$ of $G$, which is the graph obtained from $G$ by adding edges between vertices at distance $2$. The {\em $2$-distance chromatic number} of $G$ is the minimum number of colors in a $2$-distance coloring of $G$, and is denoted by $\chi^2(G)$.   The concept of $d$-distance coloring, where $d$ is a positive integer, was introduced in~\cite{Kramer_intro} and surveyed in~\cite{Kramer_surv}.  

A famous conjecture due to Wegner \cite{wegner} concerns the $2$-distance chromatic number of planar graphs:

\begin{conjecture}[Wegner~\cite{wegner}]
If $G$ is planar with maximum degree $\Delta$, then
\[
\chi^2(G) \leq 
\begin{cases}
7, & \Delta \leq 3, \\
\Delta + 5, & 4 \leq \Delta \leq 7, \\
\left\lfloor \dfrac{3\Delta}{2} \right\rfloor + 1, & \Delta \geq 8.
\end{cases}
\]
\label{conj:W}
\end{conjecture}

The conjecture has been proven for $\Delta=3$ (see Hartke et al.~\cite{HJT} and Thomassen~\cite{th-2018}), but remains open in general. 
Addressing weaker versions of the conjecture may help solving the original problem. 
Among coloring invariants that lie between coloring and $2$-distance coloring, one can investigate a frugal coloring, introduced in 1997 by Hind, Molloy and Reed~\cite{HMR}, and defined as follows. 
For a positive integer $t$, a proper vertex coloring of a graph $G$ is $t$-\textit{frugal} if no color appears more than $t$ times in the neighborhood of any vertex in $G$. The $t$-\textit{frugal chromatic number} of $G$ is the minimum $k$ for which $G$ admits a $t$-frugal coloring with $k$ colors, and we denote it by $\chi_t^f(G)$. Clearly, the condition of $t$-frugal coloring is weaker than that of $2$-distance coloring, and its strength depends on $t$ as described in the following chain of inequalities:
$$\chi(G)=\chi_{\Delta}^{f}(G)\leq \ldots\leq \chi_{2}^{f}(G)\leq \chi_{1}^{f}(G)=\chi^2(G).$$
In this paper, we consider a weaker version of Wegner's conjecture by using $t$-frugal colorings. Such a result would be of independent interest, but may also lead to an approach to resolve the conjecture itself.  The following problem arises naturally:

\begin{problem}
\label{p:limited1}
Prove that for every planar graph $G$ with maximum degree $\Delta\in \{4,5,6,7\}$, $$\fd(G)\le \Delta+5. $$ 
\end{problem}

Relaxations similar to frugal coloring have been fruitful:
Cho, Choi and Lidick\'{y}~\cite{ccl-2024} considered a relaxation of Wegner's conjecture for graphs with maximum degree $4$. They proved that every planar graph has a proper $9$-coloring such that at most one color is allowed to be repeated in the neighborhood of a vertex of degree $4$ (but nowhere else). Note that the condition of $2$-frugal coloring is slightly weaker (since we allow that for every vertex a color may be repeated in its neighborhood, and for vertices of degree $4$ we even allow that two colors are repeated). Hence, the result from~\cite{ccl-2024} implies that $\fd(G)\leq 9$ holds for every planar graph $G$ with maximum degree $4$. Here we improve this upper bound by proving the following statement.

\begin{theorem}
\label{thm:Delta4}
For every planar graph $G$ with $\Delta(G)\le4$, $$\fd(G)\le 7.$$ 
\end{theorem}

For planar graphs with $\Delta\ge 5$, Problem \ref{p:limited1} is widely open. In this paper, we resolve Problem \ref{p:limited1} for the case $\Delta=5$.

\begin{theorem}
\label{thm:Delta5}
For every planar graph $G$ with $\Delta(G)\le5$, $$\fd(G)\le 10.$$ 
\end{theorem}

Section~\ref{sec:Delta4} is devoted to the proof of Theorem~\ref{thm:Delta4}, and Section~\ref{sec:Delta5} contains the proof of Theorem~\ref{thm:Delta5}. In the next paragraph we give some basic definitions needed throughout the paper.

In this paper, we consider finite, simple graphs. Given a graph $G$ and a vertex $v\in V(G)$, the {\em neighborhood} of $v$ in $G$ is $N_G(v)=\{u\in V(G):\, uv\in E(G)\}$ and its {\em closed neighborhood} is $N_G[v]=N_G(v)\cup\{v\}$.
The {\em degree}, $\deg_G(v)$, of $v$ in $G$ is $|N_G(v)|$, while $\delta(G)$ and $\Delta(G)$ are the minimum, resp.\ the maximum degree of vertices in $G$. 
We denote the order of a graph $G$ by $n(G)$. For convenience, let $[n]=\{1,\ldots,n\}$.

\section{Proof of Theorem~\ref{thm:Delta4}}
\label{sec:Delta4}

The proof is by a contradiction, assuming that there exists a counter-example. Let $G$ be a minimum counter-example to the statement of the theorem with respect to the order of $G$. Thus, $G$ is a planar graph with $\Delta(G)\le 4$ such that $\fd(G)>7$, while every planar graph $G'$ with order $n(G')<n(G)$ has $\fd(G')\le 7$. First, we will prove two claims, which restrict the structure of the counter-example $G$, and then, by using a simple discharging argument, we will obtain a contradiction with $G$ being planar.

The following proof technique will be used in all proofs of our claims.
A set $X$, where $X\subset V(G)$, will be removed from $G$ so that $G'=G-X$ has smaller order and thus admits a $2$-frugal coloring $c$. Our goal is to extend $c$ from $G'$ to $G$, and while doing so, some of the colors cannot be used, and we call them forbidden colors. If $x\in X$, then the colors $\{c(y):\, y\in N_G(x)\cap V(G')\}$ are forbidden for $x$, since we want $c$ to be a proper coloring of $G$. In addition, for every $y\in N_G(x)$, any eventual color $\alpha\in [7]$ that appears twice in $N_G(y)\cap V(G')$ is forbidden for $x$, since assigning $c(x)=\alpha$ would result in the neighborhood of $y$ in $G$ having three vertices with color $\alpha$. Any color that is not forbidden for $x$ is called a free color for $x$. Now, assigning a free color to every vertex of $X$ implies that each vertex $y$ in $G'$ will be properly colored with respect to $2$-frugal coloring condition; notably neighbors of $y$ in $G$ will receive a different color from $c(y)$ and each color will appear at most two times in $N_G(y)$. However, to ensure that each vertex $x$ in $X$ will be properly colored with respect to $2$-frugal coloring condition we need to use a free color for every such vertex and additionally prevent ``collisions'' between vertices in $X$. Thus, while coloring vertices of $X$ one at a time, the number of free colors of not yet colored vertices in $X$ may need to be further reduced. 

\begin{claim} $\delta(G)\ge 4$.
\label{cl:1.1}
\end{claim} 
\begin{proof}
Suppose $x$ is a vertex of degree $3$ in $G$ and let $N_G(x)=\{x_1,x_2,x_3\}$. Let $X_i=N_G(x_i)\setminus\{x\}$, and note that $|X_i|\le 3$ for all $i\in [3]$. Consider $G'$ obtained from $G-x$ by possibly adding an edge between $x_1$ and $x_2$ (if $x_1x_2\notin E(G)$). Since $G'$ has smaller order than $G$ and $\Delta(G')\le 4$, there exists a $2$-frugal coloring $c$ of $G'$ using 7 colors. Note that one can extend $c$ to a proper coloring of $G$ by assigning to $x$ a color not in $\{c(x_1),c(x_2),c(x_3)\}$. Now, $c$ is also $2$-frugal, unless $c(x)$ coincides with a color that appears twice in some $X_i$. Note that there are at most three forbidden colors that arise from $|N_G(x)|=3$, and for each $i\in [3]$ there is only one forbidden color that arises from $|X_i|\le 3$ and potential repetitions of colors in $N_G(x_i)$. Altogether, we infer there are at most $3+3\times 1=6$ forbidden colors for the vertex $x$. (See also Fig.~\ref{fig:cl:1.1} depicting a possible situation that can appear with the coloring in $G'$; note that in each neighborhood $X_i$ there is a vertex whose color does not influence the resulting color for $x$ and thus the colors of these vertices are omitted in the picture.) Thus, there is at least one free color for $x$, and so $c$ can be extended to $G$, resulting in a $2$-frugal coloring of $G$, a contradiction.
\end{proof}

\begin{figure}[ht!]
\begin{center}
	\begin{tikzpicture}[scale=1]
	\draw (-2,0)--(0,0) -- (2,0);
	\draw (0,0)--(0,2);
        \draw (3,-1)--(2,0)--(3,0);
        \draw (2,0)--(3,1);
         \draw (-3,-1)--(-2,0)--(-3,0);
        \draw (-2,0)--(-3,1);
         \draw (-1,3)--(0,2)--(0,3);
        \draw (0,2)--(1,3);
         \draw[dotted, line width=1.5pt] (0,2)--(2,0);
	
    
	\draw [fill=black] (2, 0) circle (2.5pt) node[above] {$1$} node[below] {$x_2$};
    \draw [fill=white] (0, 0) circle (2.5pt);
	\draw [fill=black] (-2, 0) circle (2.5pt) node[above] {$3$} node[below] {$x_3$};
	\draw [fill=black] (0, 2) circle (2.5pt) node[left] {$2$} node[right] {$x_1$};
	\draw [fill=black] (3, 0) circle (2.5pt) node[above] {$4$};
    \draw [fill=black] (3, 1) circle (2.5pt) node[above] {$4$};
       \draw [fill=black] (-3, -1) circle (2.5pt) node[above] {$5$};
      \draw [fill=black] (-3, 0) circle (2.5pt) node[above] {$5$};
\draw [fill=black] (-3, 1) circle (2.5pt) ;
\draw [fill=black] (3, -1) circle (2.5pt) ;
\draw [fill=black] (-1, 3) circle (2.5pt) ;
    \draw [fill=black] (1, 3) circle (2.5pt) node[above] {$6$};
	\draw [fill=black] (0,3) circle (2.5pt) node[above] {$6$};
       \draw (0,-0.2) node{$x$};

\draw[dashed] (0,0) circle (1);
	\end{tikzpicture} \\
	\caption{A possible coloring of some vertices in $G'$ in Claim~\ref{cl:1.1}. 
    Dashed circle describes the removal of $x$, while the dotted edge is added in $G'$ (if it is not already in $G$). Note that in this example, there is exactly one free color at $x$.}\label{fig:cl:1.1} 
\end{center}
\end{figure}
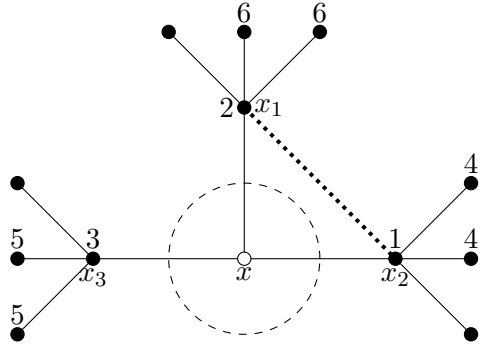

\begin{claim} There is no triangle in $G$.
\label{cl:1.2}
\end{claim} 
\begin{proof}
Suppose to the contrary that $x,y,z$ are the vertices of a triangle in $G$. By Claim \ref{cl:1.1}, $G$ is $4$-regular. Hence $\deg_G(x)=\deg_G(y)=\deg_G(z)=4$. Consider $G''=G-\{x,y,z\}$, and let $N_G(x) \cap V(G'')=\{x_1,x_2\}$, $N_G(y) \cap V(G'')=\{y_1,y_2\}$, and $N_G(z) \cap V(G'')=\{z_1,z_2\}$. Further let $X_i=N_G(x_i)\setminus\{x\}$, $Y_i=N_G(y_i)\setminus\{y\}$ and $Z_i=N_G(z_i)\setminus\{z\}$ for all $i \in [2]$. Now let $G'$ be obtained from graph $G''$ by adding the edges $x_1x_2$, $y_1y_2$ and $z_1z_2$ if they are not already in $G$ (see Fig. \ref{fig:cl:1.2}). Since $G'$ has smaller order than $G$ and $\Delta(G')\le 4$, there exists a $2$-frugal coloring $c$ of $G'$ using $7$ colors.

In the next step we count the number of free colors at vertices $x$, $y$ and $z$ in $G$. Note that $|X_i| \le 3$ for each $i\in [2]$, which together with colors $c(x_1)$ and $c(x_2)$ adds up to at most $2 \times 1 + 2=4$ forbidden colors for $x$. Since $y_1y_2 \in E(G')$ and $z_1z_2 \in E(G')$, we have $c(y_1) \ne c(y_2)$ and $c(z_1) \ne c(z_2)$, thus regardless of the color assigned to $x$, no three vertices in $N_G(y)$ or $N_G(z)$ will receive the same color. Hence, $x$ indeed has at least three free colors. By symmetry, we derive that 
 $y$ as well as $z$ has at least three free colors. 
Therefore, one can easily extend $c$ to a $2$-frugal coloring of $G$, a contradiction. 
\end{proof}

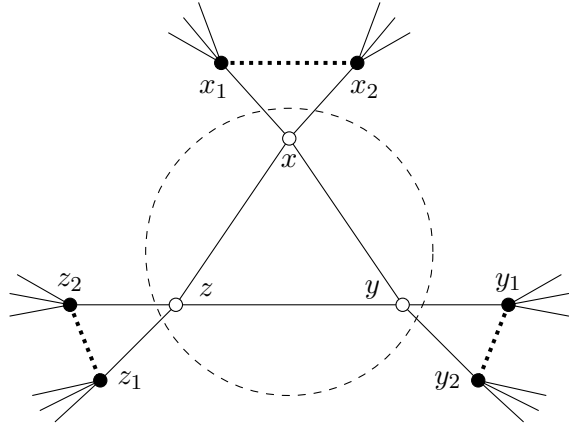
\begin{figure}[ht!]
\begin{center}
	\begin{tikzpicture}[scale=1]
	\draw (-1.5,0)--(1.5,0)--(0,2.2)--(-1.5,0);
    \draw (-1.5,0)--(-2.9,0);
	     \draw (-1.5,0)--(-2.5,-1);
	  \draw (1.5,0)--(2.9,0);
	     \draw (1.5,0)--(2.5,-1);
         \draw[dotted, line width=1.5pt] (2.9,0)--(2.5,-1);
       \draw[dotted, line width=1.5pt] (-2.9,0)--(-2.5,-1);

        \draw (-3.7,0.15)--(-2.9,0);
         \draw (-3.6,0.4)--(-2.9,0);
     \draw (-3.7,-0.15)--(-2.9,0);

         \draw (-3.4,-1.2)--(-2.5,-1);
     \draw (-3.3,-1.4)--(-2.5,-1);
         \draw (-3.1,-1.55)--(-2.5,-1);

            \draw (3.7,0.15)--(2.9,0);
         \draw (3.6,0.4)--(2.9,0);
     \draw (3.7,-0.15)--(2.9,0);

         \draw (3.4,-1.2)--(2.5,-1);
     \draw (3.3,-1.4)--(2.5,-1);
         \draw (3.1,-1.55)--(2.5,-1);

 \draw (0,2.2)--(-0.9,3.2);
   \draw (0,2.2)--(0.9,3.2);
         \draw[dotted, line width=1.5pt] (-0.9,3.2)--(0.9,3.2);

          \draw (-1.2,4)--(-0.9,3.2);
 \draw (-1.4,3.8)--(-0.9,3.2);
          \draw (-1.6,3.6)--(-0.9,3.2);
                      
  \draw (1.2,4)--(0.9,3.2);
 \draw (1.4,3.8)--(0.9,3.2);
          \draw (1.6,3.6)--(0.9,3.2);


    	\draw [fill=white] (1.5, 0) circle (2.5pt);
    \draw [fill=white] (-1.5, 0) circle (2.5pt);
    \draw [fill=white] (0, 2.2) circle (2.5pt);
    \draw [fill=black] (2.9, 0) circle (2.5pt);
    \draw [fill=black] (2.5, -1) circle (2.5pt);
     \draw [fill=black] (-2.9, 0) circle (2.5pt);
    \draw [fill=black] (-2.5, -1) circle (2.5pt);
    \draw [fill=black] (-0.9, 3.2) circle (2.5pt);
      \draw [fill=black] (0.9, 3.2) circle (2.5pt);
        
         \draw (-1.1,0.2) node{$z$};
       \draw (1.1,0.2) node{$y$};
        \draw (0,1.9) node{$x$};
          \draw (-1,2.85) node{$x_1$};
\draw (1,2.85) node{$x_2$};
\draw (-2.9,.3) node{$z_2$};
\draw (-2.1,-1) node{$z_1$};
\draw (2.9,.3) node{$y_1$};
\draw (2.1,-1) node{$y_2$};

\draw[dashed] (0,0.7) circle (1.9);
	\end{tikzpicture} \\
	\caption{The situation in Claim~\ref{cl:1.2}. Dashed circle describes the removal of the triangle, while dotted edges are added in $G'$ (if they are not already there in $G$). }\label{fig:cl:1.2}
\end{center}
\end{figure}

We finish the proof of Theorem \ref{thm:Delta4} by using a basic discharging argument.
Letting $F(G)$ be the set of faces of $G$, we assign the initial charge $\mu(z)$ for each element $z\in V(G)\cup F(G)$ to be $\deg_G(z)-4$ (note that the degree of a face is defined as the number of its boundary edges (or vertices)). By Euler's formula the initial total charge of the graph, which is the sum of all charges, equals
$$\sum_{v\in V(G)}{(\deg_G(v)-4)}+\sum_{f\in F(G)}{(\deg_G(f)-4)}=2|E(G)|-4|V(G)|+2|E(G)|-4|F(G)|=-8.$$

By Claims \ref{cl:1.1} and \ref{cl:1.2}, $\delta(G)=4$ and all the faces of $G$ have lengths at least $4$, we have

$$\sum_{v\in V(G)}{(\deg_G(v)-4)}+\sum_{f\in F(G)}{(\deg_G(f)-4)}\ge 0,$$ which is a contradiction. The proof is complete.

\section{Proof of Theorem~\ref{thm:Delta5}}
\label{sec:Delta5}

The proof follows similar steps and uses techniques similar to, though somewhat more complex than, those in the proof of Theorem \ref{thm:Delta4}. Suppose to the contrary that there exists a counter-example to the theorem, and let $G$ be a counter-example with minimum order. Thus, $G$ is a planar graph with $\Delta(G)\le 5$ such that $\fd(G)>10$, yet every planar graph $G'$ with order $n(G')<n(G)$ has $\fd(G')\le 10$. In the sequel, we will prove a sequence of claims, which restrict the structure of the counter-example $G$. In the end, by using a discharging argument we will obtain a contradiction with $G$ being planar. 


\begin{claim} $\delta(G)\ge 4$.
\label{cl:1}
\end{claim} 
\begin{proof}
Suppose $x$ is a vertex of degree $3$ in $G$ and let $N_G(x)=\{x_1,x_2,x_3\}$. Let $Y_i=N(x_i)\setminus\{x\}$, and note that $|Y_i|\le 4$ for all $i\in [3]$. Consider $G'$ obtained from $G-x$ by possibly adding an edge between $x_1$ and $x_2$ (if $x_1x_2\notin V(G)$). Since $G'$ has smaller order than $G$, there exists a $2$-frugal coloring $c$ of $G'$ using 10 colors. Note that one can extend $c$ to a proper coloring of $G$ by assigning to $x$ a color not in $\{c(x_1),c(x_2),c(x_3)\}$. Now, $c$ is also $2$-frugal, unless $c(x)$ coincides with a color that appears twice in some $Y_i$. Note that there are at most three forbidden colors that arise from $|N_G(x)|=3$, and for each $i\in [3]$ there are most two forbidden colors that arise from $|Y_i|=4$ and potential repetitions of colors in $N_G(x_i)$. Altogether, we infer there are at most $3+3\times 2=9$ forbidden colors for $c(x)$ (see also Fig.~\ref{fig:cl:1} depicting a possible situation that can appear with the coloring in $G'$). Thus, there is at least one free color for $x$, and so $c$ can be extended to $G$, resulting in a $2$-frugal coloring of $G$ (using ten colors), a contradiction.
\end{proof}

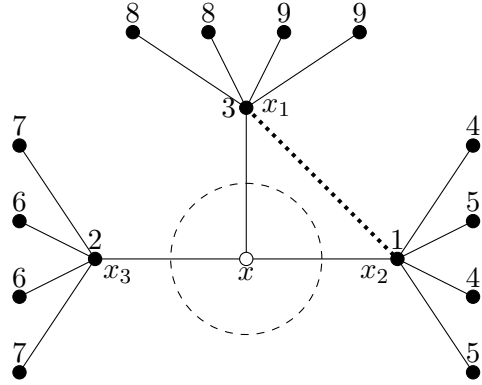
\begin{figure}[ht!]
\begin{center}
	\begin{tikzpicture}[scale=1]
	\draw (-2,0)--(0,0) -- (2,0);
	\draw (0,0)--(0,2);
        \draw (3,-0.5)--(2,0)--(3,-1.5);
        \draw (3,0.5)--(2,0)--(3,1.5);
         \draw (-3,-0.5)--(-2,0)--(-3,-1.5);
        \draw (-3,0.5)--(-2,0)--(-3,1.5);
         \draw (-1.5,3)--(0,2)--(-0.5,3);
        \draw (1.5,3)--(0,2)--(0.5,3);
         \draw[dotted, line width=1.5pt] (0,2)--(2,0);
	
    
	\draw [fill=black] (2, 0) circle (2.5pt) node[above] {$1$}; 
    \draw [fill=white] (0, 0) circle (2.5pt);
	\draw [fill=black] (-2, 0) circle (2.5pt) node[above] {$2$}; 
	\draw [fill=black] (0, 2) circle (2.5pt) node[left] {$3$}; 
	\draw [fill=black] (3, -0.5) circle (2.5pt) node[above] {$4$};
\draw [fill=black] (3, 0.5) circle (2.5pt) node[above] {$5$};
       \draw [fill=black] (-3, -0.5) circle (2.5pt) node[above] {$6$};
      \draw [fill=black] (-3, 0.5) circle (2.5pt) node[above] {$6$};
\draw [fill=black] (-3, 1.5) circle (2.5pt) node[above] {$7$};
\draw [fill=black] (-3, -1.5) circle (2.5pt) node[above] {$7$};
\draw [fill=black] (3, 1.5) circle (2.5pt) node[above] {$4$};
\draw [fill=black] (3, -1.5) circle (2.5pt) node[above] {$5$};
\draw [fill=black] (-1.5, 3) circle (2.5pt) node[above] {$8$};
	\draw [fill=black] (-0.5,3) circle (2.5pt) node[above] {$8$};
    \draw [fill=black] (1.5, 3) circle (2.5pt) node[above] {$9$};
	\draw [fill=black] (0.5,3) circle (2.5pt) node[above] {$9$};
       \draw (0,-0.2) node{$x$};
       \draw (0.4,2) node{$x_1$};
       \draw (1.7,-0.2) node{$x_2$};
       \draw (-1.7,-0.2) node{$x_3$};

\draw[dashed] (0,0) circle (1);
	\end{tikzpicture} \\
	\caption{A possible coloring of some vertices in $G'$ in Claim~\ref{cl:1}; there is only one free color at $x$.}\label{fig:cl:1}
\end{center}
\end{figure}

\begin{claim} Every triangle in $G$ has at least two vertices of degree $5$ in $G$.
\label{cl:2}
\end{claim} 
\begin{proof}
Let $x,y,z$ be the vertices of a triangle in $G$, and let us assume that $\deg_G(x)=5$ and $\deg_G(y)=\deg_G(z)=4$. (If the degrees of $y$ or $z$ in $G$ are less than $4$, then the proof follows similar lines while there are even more free colors for some vertices.) Consider $G''=G-\{x,y,z\}$, and let $N_G(x) \cap V(G'')=\{x_1,x_2,x_3\}$, $N_G(y) \cap V(G'')=\{y_1,y_2\}$, and $N_G(z) \cap V(G'')=\{z_1,z_2\}$. Further let $X_i=N_G(x_i)\setminus\{x\}$ for all $i\in[3]$, and $Y_i=N_G(y_i)\setminus\{y\}$, $Z_i=N_G(z_i)\setminus\{z\}$ for all $i \in [2]$. Now let $G'$ be obtained from graph $G''$ by adding the edges $x_1x_2$, $y_1y_2$ and $z_1z_2$ if necessary (that is, if these edges are not already in $G$; see Fig. \ref{fig:cl:2}). Since $G'$ has smaller order than $G$ and $\Delta(G)\le 5$, there exists a $2$-frugal coloring $c$ of $G'$ using $10$ colors.

In the next step we count the number of free colors at vertices $x$, $y$ and $z$ in $G$. Note that $|X_i| \le 4$ for each $i\in [3]$, which together with colors $c(x_1)$, $c(x_2)$ and $c(x_3)$ adds up to at most $2 \times 3 + 3=9$ forbidden colors for $x$. Since $y_1y_2 \in E(G')$ and $z_1z_2 \in E(G')$, we have $c(y_1) \ne c(y_2)$ and $c(z_1) \ne c(z_2)$, thus regardless of the color assigned to $x$, no three vertices in $N_G(y)$ or $N_G(z)$ will receive the same color. Hence, $x$ indeed has at least one free color. We now focus on the vertex $y$. Note that $|Y_i| \le 4$ for each $i\in [2]$, which together with colors $c(y_1)$ and $c(y_2)$ adds up to at most $2 \times 2 +2=6$ forbidden colors. However, it is possible that either $c(x_3)=c(x_1)$ or $c(x_3)=c(x_2)$, which potentially yields an additional forbidden color for $y$. Thus, altogether, $y$ has at least three free colors. Similarly, $z$ also has at least three free colors.

Next, we color the vertices of the triangle. First we color the vertex $x$ with a free color. Consequently the number of free colors at vertices $y$ and $z$ can drop by one, resulting in at least two free colors at both vertices $y$ and $z$. In the next step, we color the vertex $y$ with a free color, again potentially dropping the number of free colors at $z$ to one. Thus, we can color $z$ in the final step, which gives a $2$-frugal coloring of $G$, a contradiction, which yields the statement.
\end{proof}

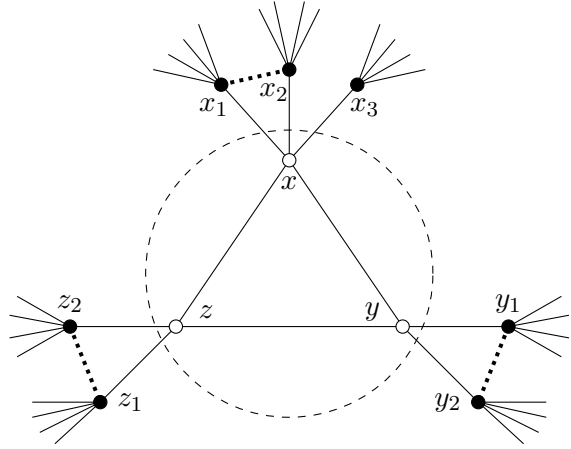
\begin{figure}[ht!]
\begin{center}
	\begin{tikzpicture}[scale=1]
	\draw (-1.5,0)--(1.5,0)--(0,2.2)--(-1.5,0);
    \draw (-1.5,0)--(-2.9,0);
	     \draw (-1.5,0)--(-2.5,-1);
	  \draw (1.5,0)--(2.9,0);
	     \draw (1.5,0)--(2.5,-1);
         \draw[dotted, line width=1.5pt] (2.9,0)--(2.5,-1);
       \draw[dotted, line width=1.5pt] (-2.9,0)--(-2.5,-1);

        \draw (-3.7,0.15)--(-2.9,0);
         \draw (-3.6,0.4)--(-2.9,0);
     \draw (-3.7,-0.15)--(-2.9,0);
         \draw (-3.6,-0.4)--(-2.9,0);

    \draw (-3.4,-1)--(-2.5,-1);
         \draw (-3.4,-1.2)--(-2.5,-1);
     \draw (-3.3,-1.4)--(-2.5,-1);
         \draw (-3.1,-1.55)--(-2.5,-1);

            \draw (3.7,0.15)--(2.9,0);
         \draw (3.6,0.4)--(2.9,0);
     \draw (3.7,-0.15)--(2.9,0);
         \draw (3.6,-0.4)--(2.9,0);

    \draw (3.4,-1)--(2.5,-1);
         \draw (3.4,-1.2)--(2.5,-1);
     \draw (3.3,-1.4)--(2.5,-1);
         \draw (3.1,-1.55)--(2.5,-1);

 \draw (0,2.2)--(-0.9,3.2);
  \draw (0,2.2)--(0,3.4);
   \draw (0,2.2)--(0.9,3.2);
         \draw[dotted, line width=1.5pt] (-0.9,3.2)--(0,3.4);

          \draw (-1.2,4)--(-0.9,3.2);
 \draw (-1.4,3.8)--(-0.9,3.2);
          \draw (-1.6,3.6)--(-0.9,3.2);
                      \draw (-1.8,3.4)--(-0.9,3.2);
                      
  \draw (1.2,4)--(0.9,3.2);
 \draw (1.4,3.8)--(0.9,3.2);
          \draw (1.6,3.6)--(0.9,3.2);
                      \draw (1.8,3.4)--(0.9,3.2);

                        \draw (-0.2,4.3)--(0,3.4);
 \draw (0.2,4.3)--(0,3.4);
          \draw (0.4,4.2)--(0,3.4);
                      \draw (-0.4,4.2)--(0,3.4);

    	\draw [fill=white] (1.5, 0) circle (2.5pt);
    \draw [fill=white] (-1.5, 0) circle (2.5pt);
    \draw [fill=white] (0, 2.2) circle (2.5pt);
    \draw [fill=black] (2.9, 0) circle (2.5pt);
    \draw [fill=black] (2.5, -1) circle (2.5pt);
     \draw [fill=black] (-2.9, 0) circle (2.5pt);
    \draw [fill=black] (-2.5, -1) circle (2.5pt);
    \draw [fill=black] (-0.9, 3.2) circle (2.5pt);
      \draw [fill=black] (0.9, 3.2) circle (2.5pt);
        \draw [fill=black] (0, 3.4) circle (2.5pt);
        
         \draw (-1.1,0.2) node{$z$};
       \draw (1.1,0.2) node{$y$};
        \draw (0,1.9) node{$x$};
          \draw (-1,2.9) node{$x_1$};
\draw (-0.2,3.1) node{$x_2$};
\draw (1,2.9) node{$x_3$};
\draw (-2.9,.3) node{$z_2$};
\draw (-2.1,-1) node{$z_1$};
\draw (2.9,.3) node{$y_1$};
\draw (2.1,-1) node{$y_2$};

\draw[dashed] (0,0.7) circle (1.9);
	\end{tikzpicture} \\
	\caption{The situation in Claim~\ref{cl:2}. Dashed circle describes the removal of the triangle, while dotted edges are added in $G'$ (if they were not already there in $G$). }\label{fig:cl:2}
\end{center}
\end{figure}

\begin{claim} \label{cl:3}
Every vertex of a diamond in $G$ has degree $5$ in $G$.
\end{claim} 
\begin{proof}
Let $x,y,z,z'$, where $xy \notin E(G)$, be the vertices of a diamond $D$ in $G$, and suppose that one of those vertices has degree $4$. We split the proof into two subcases.

First, let $\deg_G(y)=4$ and $\deg_G(x)=\deg_G(z)=\deg_G(z')=5$. (If the degrees of $y$, $z$ or $z'$ in $G$ are less than $5$, the proof follows similar lines while there may be even more free colors for some vertices.) Consider $G''=G-D$, and let $N_G(x) \cap V(G'')=\{x_1,x_2,x_3\}$, $N_G(y) \cap V(G'')=\{y_1,y_2\}$, $N_G(z) \cap V(G'')=\{z_1,z_2\}$, and $N_G(z') \cap V(G'')=\{z'_1,z'_2\}$. Further let $X_i=N_G(x_i)\setminus \{x\}$ for all $i\in[3]$, and $Y_i=N_G(y_i)\setminus \{y\}$, $Z_i=N_G(z_i)\setminus \{z\}$ and $Z'_i=N_G(z'_i)\setminus \{z'\}$ for all $i \in [2]$. Now let $G'$ be obtained from graph $G''$ by adding the edges $x_1x_2$, $y_1y_2$, $z_1z_2$ and $z'_1z'_2$ if necessary (see Fig.~\ref{fig:cl:3}$(a)$). Since $G'$ has smaller order than $G$ and $\Delta(G')\le 5$, there exists a $2$-frugal coloring $c$ of $G'$ using $10$ colors.

Let us count the number of free colors at vertices $x$, $y$, $z$ and $z'$ in $G$. Note that $|X_i| \le 4$ for each $i\in [3]$, which together with colors $c(x_1)$, $c(x_2)$ and $c(x_3)$ adds up to at most $3 \times 2 + 3=9$ forbidden colors at $x$. Since $z_1z_2 \in E(G')$ and $z'_1z'_2 \in E(G')$, we have $c(z_1) \ne c(z_2)$ and $c(z'_1) \ne c(z'_2)$, thus regardless of the color assigned to $x$, no three vertices in $N_G(z)$ and $N_G(z')$ will receive the same color. Hence, $x$ indeed has at least one free color. We now focus on the vertex $y$. Note that $|Y_i| \le 4$ for each $i\in [2]$, which together with colors $c(y_1)$ and $c(y_2)$ adds up to at most $2 \times 2 +2=6$ forbidden colors. By the same reasoning as for the vertex $x$, since $c(z_1) \ne c(z_2)$ and $c(z'_1) \ne c(z'_2)$, we get no more forbidden colors for $y$. Hence $y$ has four free colors. Finally, let us count the forbidden colors for $z$ and $z'$. Note that $|Z_i| \le 4$ for each $i\in [2]$, which together with colors $c(z_1)$ and $c(z_2)$ adds up to at most $2 \times 2 +2=6$ forbidden colors. Since it is possible that either $c(x_3)=c(x_1)$ or $c(x_3)=c(x_2)$, this potentially yields an additional forbidden color for $z$. Thus, altogether, $z$ has at least three free colors. Similarly, by symmetry, $z'$ has at least three free colors.

Next, we color the vertices of the diamond. First, color the vertex $x$ with a free color. Now the number of free colors at vertices $z$ and $z'$ can drop by one, resulting in at least two free colors at each $z$ and $z'$. Moreover, the color assigned to $x$ can coincide with a color already used in $N_G(z)$ and/or $N_G(z')$, and therefore becomes forbidden for $y$. Hence, also the number of free colors at $y$ can drop to three. In the next two steps we color the vertices $z$ and $z'$ in any order, which can be done since both have two free colors. This can result in two more forbidden colors for $y$, leaving at least one free color at $y$. Thus, $y$ can also be colored, which gives a $2$-frugal coloring of $G$, a contradiction.

\medskip

In the second case, let $\deg_G(z')=4$ and $\deg_G(x)=\deg_G(y)=\deg_G(z)=5$. (If the degrees of $x$, $y$ or $z$ in $G$ are less than $4$, then the proof again follows similar lines while there can be even more free colors for some vertices.) Consider $G''=G-z'$, and let $N_G(x) \cap V(G'')=\{x_1,x_2,x_3,z\}$, $N_G(y) \cap V(G'')=\{y_1,y_2,y_3,z\}$, $N_G(z) \cap V(G'')=\{z_1,z_2,x,y\}$, and $N_G(z') \cap V(G'')=\{z'_1,x,y,z\}$. Further let $X_i=N_G(x_i)\setminus \{x\}$ and $Y_i=N_G(y_i)\setminus \{y\}$ for all $i\in[3]$, $Z_i=N_G(z_i)\setminus \{z\}$ for all $i \in [2]$, and $Z'_1=N_G(z'_1)\setminus \{z'\}$. Now let $G'$ be obtained from graph $G''$ by adding the edge  $xy$ (see Fig. \ref{fig:cl:3}$(b)$). 
Since $G'$ has smaller order than $G$ and $\Delta(G')\le 5$, there exists a $2$-frugal coloring $c$ of $G'$ using $10$ colors.

Note that the colors at $x$, $y$ and $z$ are pairwise different since they form a triangle in $G'$, hence there are at most two vertices in $N_G(z')$ given the same color by $c$. 
Now, we count the number of free colors at vertex $z'$.
Note that $|Z'_1| \le 4$, which together with color $c(z'_1)$ adds up to at most $2+1=3$ forbidden colors. Also, we need to forbid the colors given to $x$, $y$ and $z$, since they are neighbors of $z'$ in $G$ yielding at most six forbidden colors at $z'$. Additionally, it is possible that two of the vertices $x_1$, $x_2$ and $x_3$ received the same color in $G'$, and also two of the vertices $y_1$, $y_2$ and $y_3$ received the same color in $G'$, which potentially yields two more forbidden colors for $z'$. Finally, if $c(z_1)=c(z_2)$, then this yields another forbidden color for $z'$. Altogether, at most nine colors are forbidden at $z'$. Thus, $c$ can be extended to $G$, resulting in a $2$-frugal coloring of $G$, a contradiction.  
\end{proof}

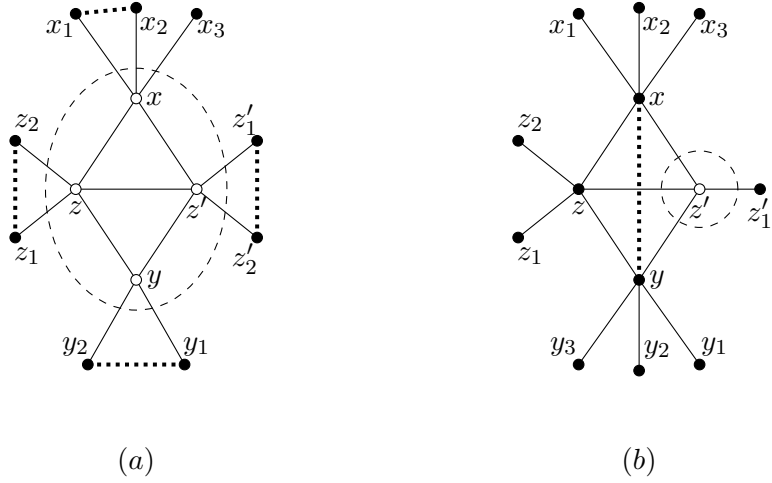
\begin{figure}[ht!]
\begin{center}
\begin{minipage}{0.45\textwidth}  
    \hskip 1.5cm
    \begin{tikzpicture}[scale=0.8]

    \draw (-1,2.9)--(0,1.5)--(0,3);
    \draw (0,1.5)--(1,2.9);
    \draw (-2,0.8)--(-1,0)--(0,1.5)--(1,0)--(0,-1.5)--(-1,0)--(-2,-.8);
    \draw (-1,0)--(1,0)--(2,0.8);
    \draw (1,0)--(2,-0.8);
    \draw (-0.8,-2.9)--(0,-1.5)--(0.8,-2.9);
       
    \draw[dotted, line width=1.5pt] (-1,2.9)--(0,3);
    \draw[dotted, line width=1.5pt] (2,0.8)--(2,-0.8);
    \draw[dotted, line width=1.5pt] (-0.8,-2.9)--(0.8,-2.9);
    \draw[dotted, line width=1.5pt] (-2,0.8)--(-2,-0.8);    
         
    \draw [fill=white] (1, 0) circle (2.5pt);
    \draw [fill=white] (-1, 0) circle (2.5pt);
    \draw [fill=white] (0, 1.5) circle (2.5pt);
    \draw [fill=white] (0,-1.5) circle (2.5pt);
    \draw [fill=black] (-2, 0.8) circle (2.5pt);
    \draw [fill=black] (-2,-0.8) circle (2.5pt);
    \draw [fill=black] (2, 0.8) circle (2.5pt);
    \draw [fill=black] (2,-0.8) circle (2.5pt);
    \draw [fill=black] (-1, 2.9) circle (2.5pt);
    \draw [fill=black] (0, 3) circle (2.5pt);
    \draw [fill=black] (1, 2.9) circle (2.5pt);
    \draw [fill=black] (-0.8,-2.9) circle (2.5pt);
    \draw [fill=black] (0.8,-2.9) circle (2.5pt);
        
    \draw (0.3,-1.5) node{$y$};
    \draw (-1,-0.3) node{$z$};
    \draw (0.3,1.5) node{$x$};
    \draw (1,-0.3) node{$z'$};
    \draw (-1.25,2.65) node{$x_1$};
    \draw (0.3,2.7) node{$x_2$};
    \draw (1.25,2.65) node{$x_3$};
    \draw (-1,-2.6) node{$y_2$};
    \draw (1,-2.6) node{$y_1$};
    \draw (-1.8,1.1) node{$z_2$};
    \draw (-1.8,-1.1) node{$z_1$};
    \draw (1.8,1.15) node{$z'_1$};
    \draw (1.8,-1.15) node{$z'_2$};
          
    \draw[dashed] (0,0) ellipse (1.5cm and 2cm);

    \draw (0,-4.5) node {$(a)$};
\end{tikzpicture}    
\end{minipage}
%
%
\begin{minipage}{0.45\textwidth}
\begin{tikzpicture}[scale=0.8]
\hskip 1.5cm

    \draw (-1,2.9)--(0,1.5)--(0,3);
    \draw (0,1.5)--(1,2.9);
    \draw (-2,0.8)--(-1,0)--(0,1.5)--(1,0)--(0,-1.5)--(-1,0)--(-2,-.8);
    \draw (-1,0)--(1,0)--(2,0);
    \draw (-1,-2.9)--(0,-1.5)--(1,-2.9);
    \draw (0,-1.5)--(0,-3);
       
    \draw[dotted, line width=1.5pt] (0,1.5)--(0,-1.5);
         
    \draw [fill=white] (1, 0) circle (2.5pt);
    \draw [fill=black] (-1, 0) circle (2.5pt);
    \draw [fill=black] (0, 1.5) circle (2.5pt);
    \draw [fill=black] (0,-1.5) circle (2.5pt);
    \draw [fill=black] (-2, 0.8) circle (2.5pt);
    \draw [fill=black] (-2,-0.8) circle (2.5pt);
    \draw [fill=black] (2, 0) circle (2.5pt);
    \draw [fill=black] (-1, 2.9) circle (2.5pt);
    \draw [fill=black] (0, 3) circle (2.5pt);
    \draw [fill=black] (1, 2.9) circle (2.5pt);
    \draw [fill=black] (-1,-2.9) circle (2.5pt);
    \draw [fill=black] (0,-3) circle (2.5pt);
    \draw [fill=black] (1,-2.9) circle (2.5pt);
        
    \draw (0.3,-1.5) node{$y$};
    \draw (-1,-0.3) node{$z$};
    \draw (0.3,1.5) node{$x$};
    \draw (1,-0.3) node{$z'$};
    \draw (-1.25,2.65) node{$x_1$};
    \draw (0.3,2.7) node{$x_2$};
    \draw (1.25,2.65) node{$x_3$};
    \draw (-1.25,-2.6) node{$y_3$};
    \draw (0.3,-2.65) node{$y_2$};
    \draw (1.25,-2.6) node{$y_1$};
    \draw (-1.8,1.1) node{$z_2$};
    \draw (-1.8,-1.1) node{$z_1$};
    \draw (2,-0.4) node{$z'_1$};
          
    \draw[dashed] (1,0) circle (18pt);

    \draw (0,-4.5) node {$(b)$};
\end{tikzpicture}
\end{minipage}

	\caption{Two cases in the proof of Claim~\ref{cl:3}. Dashed circle describes the removal of the subgraph, while dotted edges are added in $G'$ (if they were not already there in $G$); black vertices may have additional neighbors with respect to the maximum degree. }\label{fig:cl:3}
\end{center}
\end{figure}

\begin{claim} There is no subgraph isomorphic to $K_4$ in $G$.
\label{cl:noK4}
\end{claim} 
\begin{proof}
Let $x,y,z,w$ be the vertices of a subgraph $K_4$ in $G$, and let us assume that $\deg_G(x)=\deg_G(y)=\deg_G(z)=\deg_G(w)=5$. (If some of the degrees of these vertices in $G$ are less than $5$, the proof follows similar lines while there are even more free colors for some vertices.) Consider $G''=G-V(K_4)$, and let $N_G(x) \cap V(G'')=\{x_1,x_2\}$, $N_G(y) \cap V(G'')=\{y_1,y_2\}$, $N_G(z) \cap V(G'')=\{z_1,z_2\}$ and $N_G(w) \cap V(G'')=\{w_1,w_2\}$. Further let $X_i=N_G(x_i)\setminus\{x\}$, $Y_i=N_G(y_i)\setminus\{y\}$, $Z_i=N_G(z_i)\setminus\{z\}$ and $W_i=N_G(w_i)\setminus\{w\}$ for all $i \in [2]$. Now let $G'$ be obtained from graph $G''$ by adding the edges $x_1x_2$, $y_1y_2$, $z_1z_2$ and $w_1w_2$ if necessary (see Fig.~\ref{fig:cl:noK4}). 
Since $G'$ has smaller order than $G$ and $\Delta(G') \le 5$, there exists a $2$-frugal coloring $c$ of $G'$ using $10$ colors.

In the next step we count the number of free colors at vertices $x$, $y$, $z$ and $w$ in $G$. Note that $|X_i| \le 4$ for each $i\in [2]$, which together with colors $c(x_1)$ and $c(x_2)$ adds up to at most $2 \times 2 + 2=6$ forbidden colors for $x$. Thus, $x$ has four free colors. By symmetry, the same holds true for vertices $y$, $z$ and $w$, which all have four free colors.

Finally, we color the vertices $x$, $y$, $z$ and $w$ of the subgraph $K_4$, and this can be done in any order since all have four free colors, and since all have to receive different colors, this does not violate the $2$-frugal coloring condition. This gives a $2$-frugal coloring of $G$, a contradiction.
\end{proof}

\begin{figure}[ht!]
\begin{center}
\begin{tikzpicture}[scale=0.8]
    \draw (-1,-1)--(1,-1)--(1,1)--(-1,1)--(-1,-1)--(1,1)--(1.5,2.5);
    \draw (1,1)--(2.5,1.5);
    \draw (2.5,-1.5)--(1,-1)--(1.5,-2.5);
    \draw (-2.5,-1.5)--(-1,-1)--(-1.5,-2.5);
    \draw (-2.5,1.5)--(-1,1)--(-1.5,2.5);
    \draw plot [smooth,tension=3] coordinates {(1,-1) (3.5,3.5) (-1,1)};
    
    \draw[dotted, line width=1.5pt] (1.5,2.5)--(2.5,1.5);
    \draw[dotted, line width=1.5pt] (-2.5,1.5)--(-1.5,2.5);
    \draw[dotted, line width=1.5pt] (-1.5,-2.5)--(-2.5,-1.5);
    \draw[dotted, line width=1.5pt] (1.5,-2.5)--(2.5,-1.5);
         
    \draw [fill=white] (1, -1) circle (2.5pt);
    \draw [fill=white] (-1, -1) circle (2.5pt);
    \draw [fill=white] (-1, 1) circle (2.5pt);
    \draw [fill=white] (1, 1) circle (2.5pt);
    \draw [fill=black] (1.5, 2.5) circle (2.5pt);
    \draw [fill=black] (2.5,1.5) circle (2.5pt);
    \draw [fill=black] (1.5, -2.5) circle (2.5pt);
    \draw [fill=black] (2.5, -1.5) circle (2.5pt);
    \draw [fill=black] (-1.5, -2.5) circle (2.5pt);
    \draw [fill=black] (-2.5, -1.5) circle (2.5pt);
    \draw [fill=black] (-1.5,2.5) circle (2.5pt);
    \draw [fill=black] (-2.5,1.5) circle (2.5pt);
        
    \draw (1.4,0.8) node{$x$};
    \draw (-1.4,0.8) node{$w$};
    \draw (-1.4,-0.8) node{$z$};
    \draw (0.7,-1.3) node{$y$};
    \draw (1.2,2.8) node{$x_1$};
    \draw (2.8,1.2) node{$x_2$};
    \draw (-1.2,2.8) node{$w_2$};
    \draw (-2.8,1.2) node{$w_1$};
    \draw (-1.2,-2.8) node{$z_1$};
    \draw (-2.8,-1.2) node{$z_2$};
    \draw (2.8,-1.2) node{$y_1$};
    \draw (1.2,-2.8) node{$y_2$};

    \draw[dashed] (0,0) circle (2cm);
\end{tikzpicture}

\caption{The situation in Claim~\ref{cl:noK4}. Dashed circle describes the removal of the $K_4$, while dotted edges are added in $G'$ (if they were not already there in $G$); black vertices may have additional neighbors with respect to the maximum degree. }\label{fig:cl:noK4}
\end{center}
\end{figure}
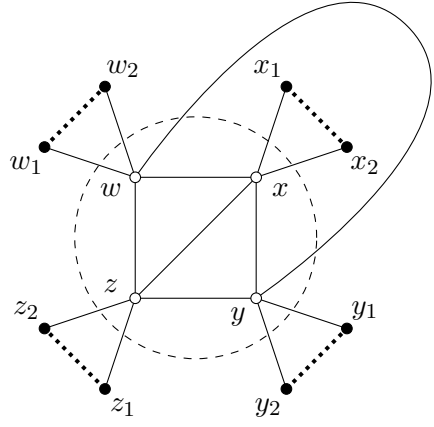

\begin{claim} Every vertex of degree $4$ in $G$ belongs to at most one triangle in $G$.
\label{cl:4}
\end{claim} 
\begin{proof}
Let $x$ be a vertex of degree $4$ in $G$ belonging to more than one triangle in $G$, and let $x_1,x_2,x_3,x_4$ denote its neighbors in $G$. Without loss of generality assume that $x_1x_2\in E(G)$. Combining Claims~\ref{cl:3} and~\ref{cl:noK4}, $x_2x_3\not\in E(G)$ and $x_1x_4\not\in E(G)$. Thus $x_3x_4\in E(G)$. Let $G'=G-x$ (see Fig. \ref{fig:cl:4}). Since $G'$ has smaller order that $G$ and $\Delta(G')\le 5$, $G'$ has a $2$-frugal coloring $c$ of $G'$ using $10$ colors.

Since every vertex $x_i$, $i\in[4]$, has three neighbors in $G-N_G[x]$ by Claim~\ref{cl:2}, there are at most $8$ forbidden colors for $x$ (together with $c(x_1)$, $c(x_2)$, $c(x_3)$ and $c(x_4)$). Thus there is a free color for $x$, implying that $c$ can be extended to a $2$-frugal coloring of $G$, a contradiction.
\end{proof}

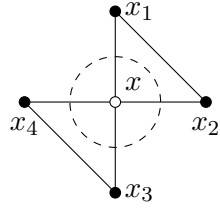
\begin{figure}[ht!]
\begin{center}
\begin{tikzpicture}[scale=0.8]
    \draw (1.5,0)--(0,1.5)--(0,-1.5)--(-1.5,0)--(1.5,0);
    
         
    \draw [fill=white] (0, 0) circle (2.5pt);
    \draw [fill=black] (1.5, 0) circle (2.5pt);
    \draw [fill=black] (0,1.5) circle (2.5pt);
    \draw [fill=black] (-1.5, 0) circle (2.5pt);
    \draw [fill=black] (0, -1.5) circle (2.5pt);
        
    \draw (0.3,0.3) node{$x$};
    \draw (0.4,1.5) node{$x_1$};
    \draw (-1.5,-0.4) node{$x_4$};
    \draw (0.4,-1.5) node{$x_3$};
    \draw (1.5,-0.4) node{$x_2$};
              
    \draw[dashed] (0,0) circle (0.75cm);
\end{tikzpicture}

\caption{The situation in Claim~\ref{cl:4}. Dashed circle describes the removal of the vertex of degree $4$; black vertices may have additional neighbors with respect to the maximum degree. }\label{fig:cl:4}
\end{center}
\end{figure}

\begin{claim} There is no subgraph isomorphic to a $4$-wheel in $G$.
\label{cl:no4w}
\end{claim}
\begin{proof}
Let $x,y,z,u,v$ be the vertices of a $4$-wheel $W$ in $G$ with the central vertex $x$. Since each of these vertices belongs to a diamond, we infer $\deg_G(x)=\deg_G(y)=\deg_G(z)=\deg_G(u)=\deg_G(v)=5$.  Consider $G''=G-V(W)$, and let $N_G(x) \cap V(G'')=\{x_1\}$, $N_G(y) \cap V(G'')=\{y_1,y_2\}$, $N_G(z) \cap V(G'')=\{z_1,z_2\}$, $N_G(u) \cap V(G'')=\{u_1,u_2\}$ and $N_G(v) \cap V(G'')=\{v_1,v_2\}$. Further let $X_1=N_G(x_1)\setminus\{x\}$, and $Y_i=N_G(y_i)\setminus\{y\}$, $Z_i=N_G(z_i)\setminus\{z\}$, $U_i=N_G(u_i)\setminus\{u\}$ and $V_i=N_G(v_i)\setminus\{v\}$ for all $i \in [2]$. Now let $G'$ be obtained from $G''$ by adding the edges $y_1y_2$, $z_1z_2$, $u_1u_2$ and $v_1v_2$ if necessary (see Fig. \ref{fig:cl:no4w}). Since $G'$ has smaller order than $G$ and $\Delta(G') \le 5$, there exists a $2$-frugal coloring $c$ of $G'$ using $10$ colors.

Let us count the number of free colors at vertices of the $4$-wheel $W$ in $G$. Note that $|X_1| \le 4$, which together with the color $c(x_1)$ adds up to at most $2 \times 1 + 1=3$ forbidden colors for $x$. Thus, $x$ has seven free colors. For the remaining vertices of $W$ we proceed in the same manner. Note that $|Y_i| \le 4$ for each $i\in [2]$, which together with colors $c(y_1)$ and $c(y_2)$ adds up to at most $2 \times 2 + 2=6$ forbidden colors for $y$. Thus, $y$ has four free colors, and the same holds true for vertices $z$, $u$ and $v$, which all have four free colors.

Finally, we color the vertices of the $4$-wheel $W$ in the following way. Firstly, we can color the vertices $y$, $z$, $u$ and $v$ with pairwise different colors in any order since they all have four free colors. This will not violate the frugal coloring condition in $N_G(x)$ even if one of those colors is the same as the color $c(x_1)$. After coloring these four vertices,  the number of free colors for $x$ can drop to three, and we can use one of those colors for $x$ to obtain a $2$-frugal coloring of $G$, which is a contradiction.
\end{proof}

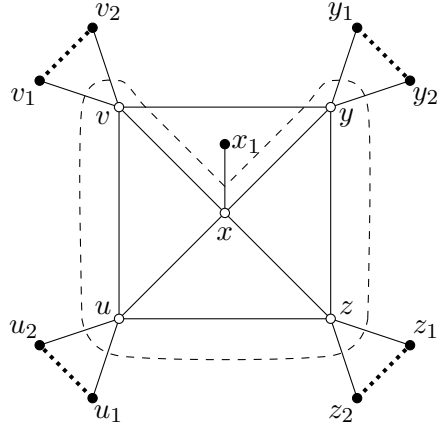
\begin{figure}[ht!]
\begin{center}
\begin{tikzpicture}[scale=0.7]
    \draw (-2,-2)--(2,-2)--(2,2)--(-2,2)--(-2,-2)--(2,2)--(2.5,3.5);
    \draw (2,2)--(3.5,2.5);
    \draw (3.5,-2.5)--(2,-2)--(2.5,-3.5);
    \draw (-3.5,-2.5)--(-2,-2)--(-2.5,-3.5);
    \draw (-3.5,2.5)--(-2,2)--(-2.5,3.5);
    \draw (-2,2)--(2,-2);
    \draw (0,0)--(0,1.3);
    \draw [dashed] plot [smooth,tension=0.5] coordinates {(0,0.5) (-0.5,1) (-1.5,2) (-2,2.5) (-2.7,2) (-2.7,-2)}; 
    \draw [dashed] plot [smooth,tension=0.5] coordinates{ (2.7,-2) (2.7,2) (2,2.5) (1.5,2) (0.5,1) (0,0.5)};
    \draw [dashed] plot [smooth,tension=0.5] coordinates { (-2.7,-2) (-2,-2.7) (2,-2.7) (2.7,-2)};
    
    \draw[dotted, line width=1.5pt] (2.5,3.5)--(3.5,2.5);
    \draw[dotted, line width=1.5pt] (-3.5,2.5)--(-2.5,3.5);
    \draw[dotted, line width=1.5pt] (-2.5,-3.5)--(-3.5,-2.5);
    \draw[dotted, line width=1.5pt] (2.5,-3.5)--(3.5,-2.5);
         
    \draw [fill=white] (0, 0) circle (2.5pt);
    \draw [fill=white] (2, -2) circle (2.5pt);
    \draw [fill=white] (-2, -2) circle (2.5pt);
    \draw [fill=white] (-2, 2) circle (2.5pt);
    \draw [fill=white] (2, 2) circle (2.5pt);
    \draw [fill=black] (2.5, 3.5) circle (2.5pt);
    \draw [fill=black] (3.5,2.5) circle (2.5pt);
    \draw [fill=black] (2.5, -3.5) circle (2.5pt);
    \draw [fill=black] (3.5, -2.5) circle (2.5pt);
    \draw [fill=black] (-2.5, -3.5) circle (2.5pt);
    \draw [fill=black] (-3.5, -2.5) circle (2.5pt);
    \draw [fill=black] (-2.5,3.5) circle (2.5pt);
    \draw [fill=black] (-3.5,2.5) circle (2.5pt);
    \draw [fill=black] (0,1.3) circle (2.5pt);
        
    \draw (0,-0.4) node{$x$};
    \draw (2.3,1.8) node{$y$};
    \draw (-2.3,1.8) node{$v$};
    \draw (-2.3,-1.8) node{$u$};
    \draw (2.3,-1.8) node{$z$};
    \draw (2.2,3.8) node{$y_1$};
    \draw (3.8,2.2) node{$y_2$};
    \draw (-2.2,3.8) node{$v_2$};
    \draw (-3.8,2.2) node{$v_1$};
    \draw (-2.2,-3.8) node{$u_1$};
    \draw (-3.8,-2.2) node{$u_2$};
    \draw (3.8,-2.2) node{$z_1$};
    \draw (2.2,-3.8) node{$z_2$};
    \draw (0.4,1.3) node{$x_1$};

\end{tikzpicture}

\caption{The situation in Claim~\ref{cl:no4w}. Dashed circle describes the removal of the $4$-wheel, while dotted edges are added in $G'$ (if they were not already there in $G$); black vertices may have additional neighbors with respect to the maximum degree. }\label{fig:cl:no4w}
\end{center}
\end{figure}

\begin{claim} There is no vertex of degree $5$ in $G$ which belongs to exactly four triangles in $G$.
\label{cl:5}
\end{claim} 
\begin{proof}
Suppose to the contrary that $x$ is such a vertex, and let $y,z,u,v, w$ be its neighbors. Without loss of generality assume that   $\{yz,zu,uv,vw\}\subset E(G)$ whereas $yw\notin E(G)$, which results in exactly four triangles to which $x$ belongs. Due to Claims~\ref{cl:noK4} and~\ref{cl:no4w}, there are no additional edges between any two vertices in $\{y,z,u,v,w\}$, and due to Claim~\ref{cl:3} the degree of each of these vertices is $5$. Let $G''=G-N_G[x]$,  and let $N_G(y)\cap V(G'')=\{y_1,y_2,y_3\}$,  $N_G(z)\cap V(G'')=\{z_1,z_2\}$, $N_G(u)\cap V(G'')=\{u_1,u_2\}$, $N_G(v)\cap V(G'')=\{v_1,v_2\}$, $N_G(w)\cap V(G'')=\{w_1,w_2,w_3\}$. Furthermore, $Y_i =N_G(y_i)\setminus\{y\}$, $W_i =N_G(w_i)\setminus\{w\}$, where $i\in[3]$, and $Z_i =N_G(z_i)\setminus\{z\}$, $U_i =N_G(u_i)\setminus\{u\}$, $V_i =N_G(v_i)\setminus\{v\}$, where $i\in [2]$. 
Let $G'$ be the graph obtained from $G''$ by adding edges $y_1y_2,z_1z_2,u_1u_2,v_1v_2, w_1w_2$ whenever they are not already in $G''$ (see Fig. \ref{fig:cl:5}). Since $G'$ has smaller order than $G$ and $\Delta(G')\le 5$, $G'$ has a $2$-frugal coloring $c$ using $10$ colors.

Let us count the numbers of free colors at vertices in $V(G)\setminus V(G')$. Note that $y$ and $w$ each have at most 
$2\times 3+ 3=9$ forbidden colors, whereas $z,u,v$ each have at most  $2\times 2+2=6$ forbidden colors. For the vertex $x$, since $\{y_1y_2,z_1z_2,u_1u_2,v_1v_2,w_1w_2\}\subset E(G')$, the only forbidden colors appear at $x$ when $c(y_i)=c(y_3)$ or $c(w_i)=c(w_3)$ for some $i\in [2]$; thus altogether there are at most two forbidden colors at $x$. 

Now we extend the coloring $c$ of $G'$ to $G$. We start by coloring the vertices $y$ and $w$, which each have at least one free color. Note that if $y$ and $w$ have a common neighbor (that is, $y_i=w_j$ for some $i$ and $j$), then actually $y$ and $w$ each have at least two free colors. In such a case, we color $y$ and $w$ with two distinct colors. Otherwise, if $y$ and $w$ do not have a common neighbor, then color each of them with its own free color. (Note that in the latter case, it is possible that $c(y)=c(w)$.)
In the rest of the proof, we consider each of these two possibilities.

First, assume that $c(y)=c(w)$. This implies possibly one more forbidden color for each of $z,u,v$ (for $z$ and $v$ since $y$ or $w$ is its neighbor, for $u$ since $x$ has two neighbors $y$ and $w$ having the same color). Thus, $z$ has at least $3$ free colors, and after coloring it, this may decrease numbers of free colors for $u$ and $v$ by one. Hence $u$ as well as $v$ has at least two free colors. Now color $u$ with a free color, which may again decrease the number of free colors at $v$ by one. Thus, $v$ has at least one free color, and we next color $v$. Finally, vertex $x$ initially had at most two forbidden colors, and in its neighborhood there are at most $4$ colors, which are additionally forbidden for $x$. Hence there are at least $4$ free colors at $x$ and we color it. Therefore $G$ has a $2$-frugal coloring using at most $10$ colors, a contradiction.

Second, assume that $c(y)\not=c(w)$. If $c(w)$ is a free color for $z$, we assign $z$ the color $c(w)$, which decreases the number of free colors at $u$ by one. Hence $u$ still has at least two free colors (indeed, from initial four free colors at $u$, the colors $c(y)$ and $c(w)$ are also forbidden for $u$). Since $v$ is a neighbor of $w$, $c(w)$ is already forbidden for $v$ so the number of free colors at $v$ after coloring of $z$ remains the same (at least two). Then we color $u$ with a free color, which decreases the number of free colors for $v$ by $1$. However, at least one free color for $v$ remains implying that we obtain a $2$-frugal coloring of $G$, a contradiction. If $c(y)$ is a free color for $v$, we are in the previous situation interchanging the roles of $z$ and $v$. Thus we may assume that $c(w)$ is not free for $z$ and $c(y)$ is not free for $v$. Vertex $u$ has at least $2$ free colors ($4$ initially minus possibly $c(w)$ and $c(y)$). We color $u$ with a free color, which decreases the numbers of free colors of $z$ and $v$ by one. For $z$ (and, symmetrically, for $v$), there are at most $8$ forbidden colors (6 initially plus $c(w)$ and $c(u)$). Color $z$ with a free color, which decreases the number of free colors at $v$ by one. Finally, we have a free color at $v$, implying that we can extend the coloring $c$ of $G'$ to a $2$-frugal coloring of $G$, a contradiction.
\end{proof}

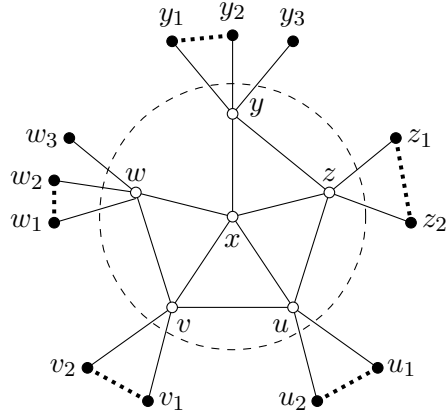
\begin{figure}[ht!]
\begin{center}
	\begin{tikzpicture}[scale=.8]
	\draw (0,1.6)--(1.6,.3);
	\draw (1.6,0.3)--(1,-1.6);
    \draw (-1.6,0.3)--(-1,-1.6)--(1,-1.6);
        
	  \draw (-1.6,0.3)--(0,-0.1)--(-1,-1.6);
    \draw (1.6,0.3)--(0,-0.1)--(1,-1.6);
    \draw (0,1.6)--(0,-0.1);
    \draw (-1,2.8)--(0,1.6)--(0,2.9);
    \draw (0,1.6)--(1,2.8);
    \draw (-2.7,1.2)--(-1.6,0.3)--(-2.95,0.5);
    \draw (-1.6,0.2)--(-2.95,-0.2);
    \draw (2.4,-2.6)--(1,-1.6)--(1.4,-3.15);
    \draw (-2.4,-2.6)--(-1,-1.6)--(-1.4,-3.15);
    \draw (2.7,1.2)--(1.6,0.3)--(2.95,-0.2);
    \draw[dotted, line width=1.5pt] (-1,2.8)--(0,2.9);
    \draw[dotted, line width=1.5pt] (2.7,1.2)--(2.95,-0.2);
    \draw[dotted, line width=1.5pt] (2.4,-2.6)--(1.4,-3.15);
    \draw[dotted, line width=1.5pt] (-2.4,-2.6)--(-1.4,-3.15);
    \draw[dotted, line width=1.5pt] (-2.95,-.2)--(-2.95,.5);
         
    \draw [fill=white] (1.6, 0.3) circle (2.5pt);
    \draw [fill=white] (-1.6, 0.3) circle (2.5pt);
    \draw [fill=white] (0, 1.6) circle (2.5pt);
    \draw [fill=white] (1,-1.6) circle (2.5pt);
    \draw [fill=white] (-1,-1.6) circle (2.5pt);
    \draw [fill=white] (0,-0.1) circle (2.5pt);
    \draw [fill=black] (1,2.8) circle (2.5pt);
    \draw [fill=black] (-1,2.8) circle (2.5pt);
    \draw [fill=black] (0,2.9) circle (2.5pt);
    \draw [fill=black] (-2.7,1.2) circle (2.5pt);
    \draw [fill=black] (-2.95,0.5) circle (2.5pt);
    \draw [fill=black] (-2.95,-0.2) circle (2.5pt);
    \draw [fill=black] (2.7,1.2) circle (2.5pt);
    \draw [fill=black] (2.95,-0.2) circle (2.5pt);
    \draw [fill=black] (-2.4,-2.6) circle (2.5pt);
    \draw [fill=black] (-1.4,-3.15) circle (2.5pt);
    \draw [fill=black] (2.4,-2.6) circle (2.5pt);
    \draw [fill=black] (1.4,-3.15) circle (2.5pt);
         
    \draw (-1.6,0.6) node{$w$};
    \draw (0.4,1.7) node{$y$};
    \draw (1.6,0.6) node{$z$};
    \draw (-0.8,-1.9) node{$v$};
    \draw (0.8,-1.9) node{$u$};
    \draw (0,-0.5) node{$x$};
    \draw (1,3.2) node{$y_3$};
    \draw (0,3.3) node{$y_2$}; 
    \draw (-1,3.2) node{$y_1$};
    \draw (-3.4,-0.2) node{$w_1$};
    \draw (-3.4,0.5) node{$w_2$}; 
    \draw (-3.15,1.2) node{$w_3$};
    \draw (2.8,-2.6) node{$u_1$};
    \draw (1,-3.15) node{$u_2$}; 
    \draw (-1,-3.15) node{$v_1$};
    \draw (-2.8,-2.6) node{$v_2$};
    \draw (3.35,-0.2) node{$z_2$}; 
    \draw (3.1,1.2) node{$z_1$};
    \draw[dashed] (0,-0.1) circle (2.2cm);
        
\end{tikzpicture} \\
    \caption{Situation from the proof of Claim~\ref{cl:5}. Dashed circle describes the removal of the subgraph, while dotted edges are added in $G'$ (if they were not already there in $G$); black vertices may have additional neighbors with respect to the maximum degree. }\label{fig:cl:5}
\end{center}
\end{figure}

\begin{claim} If a vertex of degree $5$ is the center of a $5$-wheel, then all its neighbors lie in  exactly two triangles (of the same wheel).
\label{cl:8}
\end{claim} 
\begin{proof}
Let $x,y,z,u,v,w$ be the vertices of a $5$-wheel $W$ in $G$ with the central vertex $x$. 
Note that by Claim~\ref{cl:noK4}, $W$ is an induced $5$-wheel. Moreover, by Claim~\ref{cl:3}, every vertex from $W$ lies in a diamond, hence $\deg_G(x)=\deg_G(y)=\deg_G(z)=\deg_G(u)=\deg_G(v)=\deg_G(w)=5$. 

Suppose to the contrary, that $y$ is a vertex that lies in more than two triangles in $G$. There are two situations for this to happen, and we split the proof into two subcases correspondingly.

Firstly suppose that $y$ lies in a triangle whose remaining two vertices ($y_1$ and $y_2$) do not belong to $W$. Consider the graph $G''=G-\{x,y,z,u\}$, and let $N_G(y) \cap V(G'')=\{y_1,y_2\}$, $N_G(z) \cap V(G'')=\{z_1,z_2\}$, $N_G(u) \cap V(G'')=\{u_1,u_2\}$, $N_G(v) \cap V(G'')=\{v_1,v_2\}$ and $N_G(w) \cap V(G'')=\{w_1,w_2\}$. Further let $Y_i=N_G(y_i)\setminus\{y,y_{3-i}\}$, $Z_i=N_G(z_i)\setminus\{z\}$, $U_i=N_G(u_i)\setminus\{u\}$, $V_i=N_G(v_i)\setminus\{v\}$ and $W_i=N_G(w_i)\setminus\{w\}$ for all $i \in [2]$. 
Now let $G'$ be obtained from graph $G''$ by also adding the edges $z_1z_2$ and $u_1u_2$ if necessary (see Figure~\ref{fig:5wheel}$(a)$). Since $G'$ has smaller order than $G$ and $\Delta(G') \le 5$, there exists a $2$-frugal coloring $c$ of $G'$ using $10$ colors. 

Let us count the number of free colors at vertices $x$, $y$, $z$ and $u$ in $G$. Note that $x$ is adjacent to vertices $v$ and $w$, which are already colored. This yields two forbidden colors for $x$. Moreover, if $c(v_1)=c(v_2)$ and/or $c(w_1)=c(w_2)$, this prohibits two more colors for $x$. Altogether, $x$ has at least six free colors. Next, note that $|Y_i| \le 3$ for each $i\in [2]$, which together with colors $c(y_1)$ and $c(y_2)$ adds up to at most $2 \times 1 + 2=4$ forbidden colors for $y$. However, if $c(w_1)=c(w_2)$, then, together with the color of $w$, which is adjacent to $y$, this possibly yields two more forbidden colors for $y$. Thus, $y$ has at least four free colors. On the other hand, $|U_i|\le 4$, for each $i\in [2]$, which together with colors $c(u_1)$ and $c(u_2)$ adds up to at most $2 \times 2 + 2=6$ forbidden colors for $u$.
Finally, note that $|Z_i| \le 4$ for each $i\in [2]$, which together with colors $c(z_1)$ and $c(z_2)$ adds up to at most $2 \times 2 + 2=6$ forbidden colors for $z$. 
However, if $c(v_1)=c(v_2)$, then, together with the color of $v$, which is adjacent to $u$, this possibly yields two more forbidden colors for $u$. Thus, $u$ has at least two free colors. Finally, note that $|Z_i| \le 4$ for each $i\in [2]$, which together with colors $c(z_1)$ and $c(z_2)$ adds up to at most $2 \times 2 + 2=6$ forbidden colors for $z$. Therefore, $z$ has at least four free colors. 

We color the vertices $x,y,z,u$ in the following way. Firstly, we can color the vertices $y$ and $u$ with different colors in any order since they both have at least two free colors. (Note that we can also have $c(y)=c(v)$ and $c(u)=c(w)$.) This influences the number of free colors for $x$ and $z$ that can drop by two, having four and two free colors, respectively. Finally, we color the remaining vertices $x$ and $z$, again in an arbitrary order since they both have enough free colors. Since $c(z) \ne c(y)$ and $c(z) \ne c(u)$ (even if $c(y)=c(v)$ and $c(u)=c(w)$) we obtain a $2$-frugal coloring of $G$, which is a contradiction.

\medskip

Secondly, suppose that there is a vertex $y'$ in $G$ such that $y,z,y'$ induce a triangle in $G$. Consider $G''=G-(V(W)\cup \{y'\})$. Let $N_G(y) \cap V(G'')=\{y_1\}$, $N_G(z) \cap V(G'')=\{z_1\}$, $N_G(u) \cap V(G'')=\{u_1,u_2\}$, $N_G(v) \cap V(G'')=\{v_1,v_2\}$,  $N_G(w) \cap V(G'')=\{w_1,w_2\}$, and let $N_G(y') \cap V(G'')=\{y'_1,y'_2,y'_3\}$. We further set $Y_1=N_G(y_1)\setminus\{y\}$, $Z_1=N_G(z_1)\setminus\{z\}$,
$U_i=N_G(u_i)\setminus\{u\}$, $V_i=N_G(v_i)\setminus\{v\}$ and $W_i=N_G(w_i)\setminus\{w\}$ for all $i\in[2]$, and $Y'_i=N_G(y'_i)\setminus\{y'\}$  for all $i \in [3]$. Now let $G'$ be obtained from $G''$ by adding edges $y'_1y'_2$, $u_1u_2$, $v_1v_2$ and $w_1w_2$, if necessary (see Figure~\ref{fig:5wheel}$(b)$).  Since $G'$ has smaller order than $G$ and $\Delta(G') \le 5$, there exists a $2$-frugal coloring $c$ of $G'$ using $10$ colors. 

Now we count the numbers of free colors at each of $V(G)\setminus V(G')$. Note that $x$ has no forbidden color at this moment. Since $\vert Y'_i\vert \le 4$, together with colors $c(y'_i)$, $i\in [3]$, we have at most $3 \times 2+3=9$ forbidden colors at $y'$, i.e., there is at least one free color at $y'$. For vertices $y$ and $z$, $\vert Y_1\vert \le 4$ as well as $\vert Z_1 \vert \leq 4$, which together with color $c(y_1)$, or $c(z_1)$, imply potential $3$ forbidden colors at $y$, or at $z$, respectively. In addition, the possibility that $c(y'_i)=c(y'_3)$ for some $i\in [2]$ yields another forbidden color for $y$ and $z$. 
This means that there are at least $6$ free colors at each of $y$ and $z$. For vertices $u,v,w$, each of $U_i$, $V_i$ and $W_i$, $i\in [2]$, has at most $4$ vertices, which together with colors $c(u_i)$, $c(v_i)$ or $c(w_i)$, imply at most $6$ forbidden colors at each of $u$, $v$, $w$, respectively. Hence there are at least $4$ free colors at each of $u,v,w$.

We extend the coloring $c$ of $G'$ to $G$ in the following way. We color $y'$ with a free color, which may decrease the numbers of free colors at $x,y,z,u$ and $w$ by one. Next we color $w$ with a free color, which may decrease the numbers of free colors at $x,y,z,u$ and $v$ by one. Now we color $v$ with a free color (there are still at least three such colors), which may decrease the numbers of free colors at $x,y,z$ and $u$ by one. At this moment, $u$ has still at least one free color, which we use for coloring $u$. This may decrease the numbers of free colors at $x, y$ and $z$ by one. For vertices $y$ and $z$, we have at least two free colors ($6$ initially minus $4$ colors used on vertices $y',w,v,u$). Then we can clearly color vertices $y$ and $z$ with distinct colors from their free color lists. It remains to color vertex $x$, which still has at least four free colors (from initial $10$ free colors, additional $6$ colors may become forbidden for $x$ while coloring vertices $y',w,v,u,y$ and $z$). Thus, $x$ can be colored with one of the free   colors, yielding a $2$-frugal coloring of $G$, a contradiction.
\end{proof}

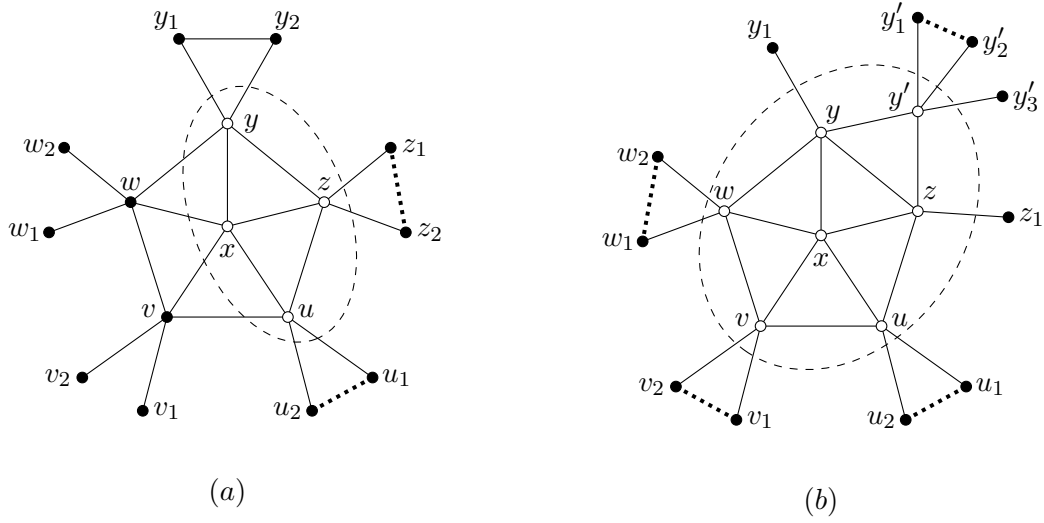
\begin{figure}[t]
\begin{center}
\begin{minipage}{0.45\textwidth}
\hskip 1cm
	\begin{tikzpicture}[scale=0.8]
	\draw (1.6,0.3)--(0,1.6)--(-1.6,.3);
	  \draw (1.6,0.3)--(1,-1.6);
    \draw (-1.6,0.3)--(-1,-1.6)--(1,-1.6);        
	\draw (-1.6,0.3)--(0,-0.1)--(-1,-1.6);
    \draw (1.6,0.3)--(0,-0.1)--(1,-1.6);
    \draw (0,1.6)--(0,-0.1);
    \draw (-0.8,3)--(0,1.6)--(0.8,3)--(-0.8,3);
    \draw (2.7,1.2)--(1.6,0.3)--(2.95,-0.2);
    \draw (-2.7,1.2)--(-1.6,0.3)--(-2.95,-0.2);
    \draw (1.4,-3.15)--(1,-1.6)--(2.4,-2.6);
    \draw (-1.4,-3.15)--(-1,-1.6)--(-2.4,-2.6);
    
  	\draw [fill=white] (1.6, 0.3) circle (2.5pt);
    \draw [fill=black] (-1.6, 0.3) circle (2.5pt);
    \draw [fill=white] (0, 1.6) circle (2.5pt);
    \draw [fill=white] (1,-1.6) circle (2.5pt);
    \draw [fill=black] (-1,-1.6) circle (2.5pt);
    \draw [fill=white] (0,-0.1) circle (2.5pt);
    \draw [fill=black] (0.8,3) circle (2.5pt);
    \draw [fill=black] (-0.8,3) circle (2.5pt);
    \draw [fill=black] (2.7,1.2) circle (2.5pt);
    \draw [fill=black] (2.95,-0.2) circle (2.5pt);
    \draw [fill=black] (-2.7,1.2) circle (2.5pt);
    \draw [fill=black] (-2.95,-0.2) circle (2.5pt);
    \draw [fill=black] (-2.4,-2.6) circle (2.5pt);
    \draw [fill=black] (-1.4,-3.15) circle (2.5pt);
    \draw [fill=black] (2.4,-2.6) circle (2.5pt);
    \draw [fill=black] (1.4,-3.15) circle (2.5pt);

    \draw (-1.6,0.6) node{$w$};
    \draw (0.4,1.6) node{$y$};
    \draw (1.6,0.6) node{$z$};
    \draw (-1.3,-1.5) node{$v$};
    \draw (1.3,-1.5) node{$u$};
    \draw (0,-0.5) node{$x$};
    \draw (1,3.3) node{$y_2$};
    \draw (-1,3.3) node{$y_1$};
    \draw (2.8,-2.6) node{$u_1$};
    \draw (1,-3.15) node{$u_2$}; 
    \draw (-1,-3.15) node{$v_1$};
    \draw (-2.8,-2.6) node{$v_2$};
    \draw (3.35,-0.2) node{$z_2$}; 
    \draw (3.1,1.2) node{$z_1$};
    \draw (-3.35,-0.2) node{$w_1$}; 
    \draw (-3.1,1.2) node{$w_2$};

\draw[dashed] [rotate around={290:(0.7,0.1)}] (0.7,0.1) ellipse (2.2cm and 1.3cm);
    \draw[dotted, line width=1.5pt] (2.4,-2.6)--(1.4,-3.15);
    \draw[dotted, line width=1.5pt] (2.95,-.2)--(2.7,1.2);

    \draw (0,-4.5) node {$(a)$};
        
	\end{tikzpicture} \\
    \end{minipage}
    \hskip 1.2cm
\begin{minipage}{0.45\textwidth}
\hskip 1cm
	\begin{tikzpicture}[scale=0.8]
        
	\draw (1.6,0.3)--(0,1.6)--(-1.6,.3);
    \draw (1.6,0.3)--(1,-1.6);
    \draw (-1.6,0.3)--(-1,-1.6)--(1,-1.6);
        
    \draw (-1.6,0.3)--(0,-0.1)--(-1,-1.6);
    \draw (1.6,0.3)--(0,-0.1)--(1,-1.6);
    \draw (0,1.6)--(0,-0.1);
    \draw (0,1.6)--(1.6,1.95)--(1.6,0.3);
    \draw (1.6,3.5)--(1.6,1.95)--(3,2.2);
    \draw (1.6,1.95)--(2.5,3.1);
    \draw (-0.8,3)--(0,1.6);
    \draw (-2.7,1.2)--(-1.6,0.3)--(-2.95,-0.2);
    \draw (1.4,-3.15)--(1,-1.6)--(2.4,-2.6);
    \draw (1.6,0.3)--(3.1,0.2);
    \draw (-1.4,-3.15)--(-1,-1.6)--(-2.4,-2.6);

    \draw [fill=white] (1.6, 0.3) circle (2.5pt);
    \draw [fill=white] (-1.6, 0.3) circle (2.5pt);
    \draw [fill=white] (0, 1.6) circle (2.5pt);
    \draw [fill=white] (1,-1.6) circle (2.5pt);
    \draw [fill=white] (-1,-1.6) circle (2.5pt);
    \draw [fill=white] (0,-0.1) circle (2.5pt);
    \draw [fill=white] (1.6,1.95) circle (2.5pt);
    \draw [fill=black] (-0.8,3) circle (2.5pt); 
    \draw [fill=black] (-2.7,1.2) circle (2.5pt); 
    \draw [fill=black] (-2.95,-0.2) circle (2.5pt); 
    \draw [fill=black] (3.1,0.2) circle (2.5pt) node[right] { $z_1$ };
    \draw [fill=black] (-2.4,-2.6) circle (2.5pt); 
    \draw [fill=black] (-1.4,-3.15) circle (2.5pt); 
    \draw [fill=black] (2.4,-2.6) circle (2.5pt); 
    \draw [fill=black] (1.4,-3.15) circle (2.5pt); 
    \draw [fill=black] (2.5,3.1) circle (2.5pt)  node[right] { $y'_2\,$ };
    \draw [fill=black] (3,2.2) circle (2.5pt) node[right] { $y'_3\,$ };
    \draw [fill=black] (1.6,3.5) circle (2.5pt) node[left] { $\,y'_1$ };

         
    \draw (1.3,2.2) node{$y'$};

    \draw (-1.6,0.6) node{$w$};
    \draw (0.2,1.9) node{$y$};
    \draw (1.8,0.6) node{$z$};
    \draw (-1.3,-1.5) node{$v$};
    \draw (1.3,-1.5) node{$u$};
    \draw (0,-0.5) node{$x$};
    \draw (2.8,-2.6) node{$u_1$};
    \draw (1,-3.15) node{$u_2$}; 
    \draw (-1,-3.15) node{$v_1$};
    \draw (-2.8,-2.6) node{$v_2$};
    \draw (-3.35,-0.2) node{$w_1$}; 
    \draw (-3.1,1.2) node{$w_2$};
    \draw (-1,3.3) node{$y_1$};

    \draw[dotted, line width=1.5pt] (1.6,3.5)--(2.5,3.1);
    \draw[dotted, line width=1.5pt] (-2.7,1.2)--(-2.95,-0.2);
    \draw[dotted, line width=1.5pt] (2.4,-2.6)--(1.4,-3.15);
    \draw[dotted, line width=1.5pt] (-2.4,-2.6)--(-1.4,-3.15);
    \draw[dashed] [rotate around={235:(0.3,0.2)}] (0.3,0.2) ellipse (2.7cm and 2.1cm);      
    \draw (0,-4.5) node {$(b)$};  

	\end{tikzpicture} \\

\end{minipage}
	\caption{Two cases in the proof of Claim~\ref{cl:8}. Dashed circle describes the removal of the subgraph, while dotted edges are added in $G'$ (if they were not already there in $G$); black vertices may have additional neighbors with respect to the maximum degree. }\label{fig:5wheel}
\end{center}
\end{figure}

\begin{claim} If there is a triangle with a vertex of degree $4$ in $G$, then each of the other two vertices of this triangle lies in at most two triangles.
\label{cl:9}
\end{claim} 
\begin{proof}
Let $x,y,z$ be the vertices of a triangle $T$ in $G$. Due to Claim~\ref{cl:2}, we may assume that $\deg_G(x)=4$ and $\deg_G(y)=\deg_G(z)=5$. Suppose to the contrary that $z$ lies in two more triangles. By Claim~\ref{cl:3} and~\ref{cl:noK4}, $z$ and $y$ lie in only one common triangle, namely $T$, and the same holds for $z$ and $x$. Therefore the additional two triangles in which $z$ lies must share an edge for otherwise $\deg_G(z)\ge 6$ which is impossible. Consider $G''=G-T$, and let $N_G(x) \cap V(G'')=\{x_1,x_2\}$, $N_G(y) \cap V(G'')=\{y_1,y_2,y_3\}$, and $N_G(z) \cap V(G'')=\{z_1,z_2,z_3\}$, where $z_1z_2, z_2z_3 \in E(G)$. Further let $X_i=N_G(x_i)\setminus\{x\}$ for all $i\in[2]$, and $Y_i=N_G(y_i)\setminus\{y\}$ for all $i\in [3]$. Letting $W=\{z_1,z_2,z_3\}$ we define $Z_i=N_G(z_i)\setminus W$ for all $i \in [3]$. Now let $G'$ be obtained from $G''$ by adding the edges $x_1x_2$, $y_1y_2$ and $z_1z_3$ if necessary (see Fig.~\ref{fig:cl:9}). Since $G'$ has smaller order than $G$ and $\Delta(G') \le 5$, there exists a $2$-frugal coloring $c$ of $G'$ using $10$ colors.

Now, count the number of free colors at vertices $x$, $y$ and $z$ in $G$. Note that $|X_i| \le 4$ for each $i\in [2]$, which together with colors $c(x_1)$ and $c(x_2)$ adds up to at most $2 \times 2 + 2=6$ forbidden colors at $x$. It can happen that either $c(y_1)=c(y_3)$ or $c(y_2)=c(y_3)$, which results in another forbidden color at $x$. Hence, there are at least three free colors at $x$. Next, note that $|Y_i| \le 4$ for each $i\in [3]$, which together with colors $c(y_1)$, $c(y_2)$, and $c(y_3)$ adds up to at most $3 \times 2 +3=9$ forbidden colors. Thus, there is at least one free color at $y$. Finally, note that $|Z_i| \le 3$ for each $i\in [3]$, and at most one color can be used twice on each $Z_i$. Together with colors $c(z_1)$, $c(z_2)$, and $c(z_3)$ this yields $3 \times 1 +3=6$ forbidden colors. Similarly as above note that either $c(y_1)=c(y_3)$ or $c(y_2)=c(y_3)$ can result in another forbidden color for $z$. Thus, there are at least three free colors at $z$.

Color the vertices of the triangle $T$ in the following way. First, color $y$ with a free color. Consequently the number of free colors at vertices $x$ and $z$ can drop by one, resulting in at least two free colors at both vertices. In the next two steps, color the vertices $x$ and $z$ with different colors in any order. This gives a $2$-frugal coloring of $G$, which is a contradiction.
\end{proof}

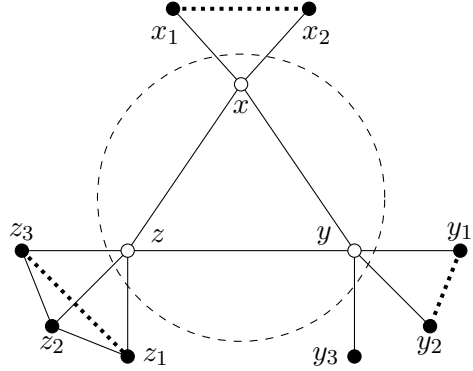
\begin{figure}[ht!]
\begin{center}
	\begin{tikzpicture}[scale=1]
	\draw (-1.5,0)--(1.5,0)--(0,2.2)--(-1.5,0);
    \draw (-1.5,0)--(-2.9,0)--(-2.5,-1)--(-1.5,-1.4)--(-1.5,0);
    \draw (-1.5,0)--(-2.5,-1);
	  \draw (1.5,0)--(2.9,0);
    \draw (1.5,0)--(2.5,-1);
    \draw (1.5,-1.4)--(1.5,0);

    \draw[dotted, line width=1.5pt] (-2.9,0)--(-1.5,-1.4);
    \draw[dotted, line width=1.5pt] (2.9,0)--(2.5,-1);





    \draw (0,2.2)--(-0.9,3.2);
    \draw (0,2.2)--(0.9,3.2);
    \draw[dotted, line width=1.5pt] (-0.9,3.2)--(0.9,3.2);

                      


   	\draw [fill=white] (1.5, 0) circle (2.5pt);
    \draw [fill=white] (-1.5, 0) circle (2.5pt);
    \draw [fill=white] (0, 2.2) circle (2.5pt);
    \draw [fill=black] (2.9, 0) circle (2.5pt) node [above] {$y_1$};
    \draw [fill=black] (2.5, -1) circle (2.5pt) node [below] {$y_2$};
    \draw [fill=black] (1.5, -1.4) circle (2.5pt) node [left] {$y_3\,$}; 
    \draw [fill=black] (-2.9, 0) circle (2.5pt) node[above] {$z_3$};
    \draw [fill=black] (-2.5, -1) circle (2.5pt) node [below] {$z_2$};
    \draw [fill=black] (-0.9, 3.2) circle (2.5pt);
    \draw [fill=black] (0.9, 3.2) circle (2.5pt);
    \draw [fill=black] (-1.5, -1.4) circle (2.5pt) node [right] {$\, z_1$};
        
         \draw (-1.1,0.2) node{$z$};
       \draw (1.1,0.2) node{$y$};
        \draw (0,1.9) node{$x$};
          \draw (-1,2.85) node{$x_1$};
\draw (1,2.85) node{$x_2$};

\draw[dashed] (0,0.7) circle (1.9);
	\end{tikzpicture} \\
	\caption{The situation in Claim~\ref{cl:9}. Dashed circle describes the removal of the triangle, while dotted edges are added in $G'$ (if they are not already there in $G$); black vertices may have additional neighbors with respect to the maximum degree.}\label{fig:cl:9}
\end{center}
\end{figure}

Now, we proceed with the proof of Theorem~\ref{thm:Delta5}. Letting $F(G)$ be the set of faces of $G$, we assign the initial charge $\mu(z)$ for each element $z\in V(G)\cup F(G)$ to be $\deg_G(z)-4$ (note that the degree of a face is defined as the number of its boundary edges (or vertices)). By Euler's formula the initial total charge of the graph, which is the sum of all charges, equals
$$\sum_{v\in V(G)}{(\deg_G(v)-4)}+\sum_{f\in F(G)}{(\deg_G(f)-4)}=2|E(G)|-4|V(G)|+2|E(G)|-4|F(G)|=-8.$$

Note that while the total charge is negative, only $3$-faces contribute to the negative values in the sum, since by Claim~\ref{cl:1}, $\delta(G)\ge 4$. 
The following discharging rules are used:
\begin{enumerate}[(R1)]
\item If $v$ is a vertex of degree $5$ that belongs to only one $3$-face, then $v$ passes the charge of ${1}$ to that $3$-face. 

\item If $v$ is a vertex of degree $5$ that belongs to two $3$-faces, then $v$ passes the charge of $\frac{1}{2}$ to the two $3$-faces to which it belongs. 

\item If $v$ is a vertex of degree $5$ that belongs to three $3$-faces, then $v$ necessarily belongs to a diamond $D$ such that $v$ has degree $3$ with respect to $D$. 
\begin{enumerate}[(R3.1)]
\item
If $v$ has degree $3$ with respect to only one diamond, then, we let $v$ pass the charge of $\frac{1}{4}$ to each of the two $3$-faces that form a diamond and the charge $\frac{1}{2}$ to the remaining $3$-face to which it belongs. 
\item 
If $v$ has degree $3$ with respect to two diamonds $D$, then these diamonds overlap in a $3$-face (the ``middle'' $3$-face), and we let $v$ pass the charge of $\frac{1}{2}$ to the middle $3$-face and the charge of $\frac{1}{4}$ to each of the remaining two $3$-faces. 
\end{enumerate}
\item If $v$ is a vertex of degree $5$ that belongs to five $3$-faces, then $v$ is the center of a $5$-wheel. 
By Claim~\ref{cl:8}, each of its neighbors $v_1,\ldots, v_5$ lies in exactly two triangles (of the same wheel). By Claim~\ref{cl:3}, each $v_i$ is a vertex of degree $5$. Then, we let each $v_i$ pass the charge of $\frac{1}{2}$ to the two $3$-faces to which it belongs (in accordance to (R2)).  
\end{enumerate}

Let $T$ be a $3$-face in $G$ with vertices $x,y$ and $z$. There are two possibilities depending on whether $T$ contains a vertex of degree $4$. 

We first consider the case when $T$ contains a vertex of degree $4$, say $\deg_G(z)=4$. Note that $T$ is the unique triangle to which $z$ belongs (by Claim~\ref{cl:4}), and the remaining two vertices $x,y$ of $T$ have degree $5$ (by Claim~\ref{cl:2}). By Claim~\ref{cl:9}, $x$ (respectively $y$) belongs to at most one additional triangle in $G$, and so by the rules (R1) and (R2), we let $x$ and $y$ each pass the charge of at least $\frac{1}{2}$ to the face $T$ (and the charge of $\frac{1}{2}$ to another $3$-face of $G$ if applicable). 

For the rest of the proof, we may assume that all vertices of $T$ have degree $5$. First, if $T$ lies in a $5$-wheel, we derive by (R4) that $T$ receives the charge $\frac{1}{2}$ from two of its vertices, which increases the charge of $T$ by $1$. 
Thus, we may further restrict to the case when $T$ is not part of $5$-wheel. 
Note that, by the rules, each of the vertices of $T$ passes the charge of at least $\frac{1}{4}$ to the $3$-face $T$. Now, if one of the vertices in $T$ lies in at most two triangles, then due to the first two rules it passes the charge of at least $\frac{1}{2}$ to $T$, which increases the charge of $T$ by at least $1$. 

Now, we are left with the case in which each of the three vertices of $T$ lies in exactly three $3$-faces (four $3$-faces is not possible by Claim~\ref{cl:5}, and five $3$-faces is not possible since no vertex of $T$ is the center of a $5$-wheel). First, assume that $T$ itself is not part of a diamond. Let $x$ be a vertex in $T$. Note that $x$ lies in at least one diamond, however $x$ has degree $3$ in exactly one diamond $D$. Hence, (R3.1) is applicable for $x$, which thus passes the charge of $\frac{1}{2}$ to $T$, and consequently the charge of $T$ increases by at least $1$.

Finally, assume that $T$ is part of a diamond. First assume that there exist two $3$-faces that share an edge with $T$, and thus form two diamonds with it; say diamond $D_1$ shares an edge with $xy$ and diamond $D_2$ shares an edge with $xz$. Then, by using the (R3.2), we infer that $x$ passes the charge of $\frac{1}{2}$ to $T$, and consequently the charge of $T$ increases by at least $1$. The second and remaining case is that exactly one $3$-face shares an edge with $T$, and forms one diamond with it; say this edge is $yz$. Since $x$ lies in three $3$-faces, where the remaining two faces do not share an edge with $T$, we infer that these two faces share an edge and form a diamond. Applying (R3.1) we infer that $x$ passes the charge of $\frac{1}{2}$ to $T$, and consequently the charge of $T$ increases by at least $1$.

All in all, we derive that $T$ received the charge of at least $1$ from its incident vertices by applying the above discharging rules. Therefore, the final total charge of the graph is non-negative, which is a contradiction with the initial charge being $-8$. The proof is complete. 
\

\section{Concluding remarks}

The next step in the context of this paper would be to resolve Problem~\ref{p:limited1} for the remaining two maximum degrees, that is, $\Delta\in \{6,7\}$. We think this may be even more involved than the case $\Delta=5$. In the light of Conjecture~\ref{conj:W}, the following problem is also natural. 

\begin{problem}
Prove that for every planar graph $G$ with maximum degree $\Delta\ge 8$, $$\fd(G)\le \Big\lfloor\frac{3\Delta}{2}\Big\rfloor+1.$$ 
\end{problem}


\section*{Acknowledgements}
B.B. and M.J. acknowledge the financial support of the Slovenian Research and Innovation Agency (research core funding No.\ P1-0297 and projects N1-0285, N1-0431 and J1-70045).

\end{document}